\documentclass[12pt]{amsart}

\usepackage[a4paper,margin=30mm]{geometry}
\usepackage[T1]{fontenc}
\usepackage[utf8]{inputenc}
\usepackage{lmodern}
\usepackage{amsmath,amssymb,amsthm,mathtools,mathrsfs}
\usepackage{microtype}
\usepackage{enumitem}
\usepackage{hyperref}



\hypersetup{
  colorlinks=true,
  linkcolor=blue,
  citecolor=blue,
  urlcolor=blue
}

\newtheorem{theorem}{Theorem}
\newtheorem{proposition}[theorem]{Proposition}
\newtheorem{lemma}[theorem]{Lemma}

\theoremstyle{definition}

\newtheorem{example}[theorem]{Example}

\theoremstyle{remark}
\newtheorem{remark}[theorem]{Remark}

\newcommand{\Tr}{\operatorname{Tr}}

\title{An infinite completely normal basis theorem}

\author{Patrik Lundstr\"om}

\address{Department of Engineering Science,
University West, SE-46186 Trollh\"attan, Sweden}
\email{patrik.lundstrom@hv.se}

\begin{document}

\begin{abstract}
Let $L/K$ be an arbitrary Galois field extension with 
group $G$, equipped with the Krull topology. 
Let $C(G,K)$ denote the $K$-vector space of continuous
functions $G\to K$, with $K$ discrete and $G$ acting
by left translation.
We prove that there exists a
$G$-equivariant $K$-linear isomorphism
$\varphi:C(G,K)\to L$ such that, for every open subgroup
$H$ of $G$, the map $\varphi^H:C(H,L^H)\to L$ induced by
extension by zero and extension of scalars is an
$H$-equivariant $L^H$-linear isomorphism. 
This gives a completely normal analogue of Lenstra's infinite
normal basis theorem. 
\end{abstract}



\maketitle

\section{Introduction}\label{sec:introduction}

Let $L/K$ be a Galois field extension with 
Galois group $G$. We consider $G$ as a topological
group equipped with the Krull 
topology; see \cite[Ch.~VI, Exercise~43]{Lang}.

When $L/K$ is \emph{finite}, the 
\emph{normal basis theorem} asserts that there 
exists an element $x \in L$ 
such that the family $(\sigma(x))_{\sigma\in G}$
forms a basis for $L$ as a vector space over $K$. In that case, 
$(\sigma(x))_{\sigma\in G}$ 
is called a 
\emph{normal basis} for $L/K$, and
$x$ is called a 
\emph{normal basis generator} for $L/K$.
For a proof,
see e.g. \cite[Ch.~VI, \S13, Thm.~13.1]{Lang}. 

Several variations of the normal basis theorem have been 
considered. Bayer-Fluckiger, Lenstra and Serre
\cite{BayerFluckiger,BayerFluckigerLenstra,BayerSerre}
studied self-dual normal bases. For finite
fields, Lenstra and Schoof \cite{LenstraSchoof} proved the primitive
normal basis theorem. Waterhouse \cite{Waterhouse} gave a unified
version of the primitive element theorem and the normal basis theorem.
Further generalizations were obtained
in \cite{BjorkholdtLundstrom}, using
polynomial expressions in the conjugates of an element, and by Biglari
\cite{Biglari}, using Zariski
methods.

Of particular interest in this article is the 
following strengthening of the normal basis theorem:

\begin{theorem}[Completely normal basis theorem]\label{thm:CNBT}
Let $L/K$ be a finite Galois field extension with Galois group $G$. 
Then there exists an element
$x\in L$ which is a normal basis generator for $L/M$ for every intermediate 
field $K\subseteq M\subseteq L$.
In that case, the family 
$(\sigma(x))_{\sigma\in G}$ is called a completely normal basis for $L/K$, and
$x$ is called a 
completely normal basis generator
for $L/K$. 
\end{theorem}

Faith~\cite{Faith} proved this theorem over infinite base fields, while Blessenohl and 
Johnsen~\cite{BlessenohlJohnsen} 
established the finite-field case.

If \([L:K]=\infty\), no \(K\)-basis of \(L\) can consist of the conjugates of a single element, since every element of \(L\) has only finitely many conjugates.
Thus, if we seek a
generalization of the normal basis theorem 
to infinite Galois extensions, we must 
look for an alternative formulation.

Suppose first that $L/K$ is finite. 
Following Lenstra~\cite{Lenstra}, we let $(G,K)$ denote
the $K$-vector space of all functions $f:G\to K$. 
We let $G$ act on $(G,K)$ by
$(\sigma f)(\tau)=f(\sigma^{-1}\tau)$,
$\sigma,\tau\in G$.
The normal basis theorem can be formulated by saying that there is a
$G$-equivariant $K$-linear isomorphism
$\varphi : (G,K) \to L$.
Indeed, if $x\in L$ is a normal basis generator for $L/K$, then
$\varphi(f) = 
\sum_{\sigma\in G}f(\sigma)\sigma(x)$.
Conversely, if $\varphi$ is such an isomorphism 
and $1_e : G \to K$ is the
characteristic function of $\{e\}$, 
where $e$ is the identity element of $G$,
then $x=\varphi(1_e)$ is a normal
basis generator for $L/K$.

Lenstra~\cite{Lenstra} showed that this formulation remains valid for
arbitrary Galois extensions if $(G,K)$ is replaced by
the $K$-vector space $C(G,K)$ of continuous functions $G\to K$, where
$G$ is equipped with the Krull topology 
and $K$ with the discrete topology. 
We equip $C(G,K)$ with the same translation action,
$(\sigma f)(\tau)=f(\sigma^{-1}\tau)$.
More precisely, Lenstra proved the following
\cite[Theorem~1]{Lenstra}:

\begin{theorem}[Infinite normal basis theorem]
\label{thm:Lenstra}
Let $L/K$ be a Galois field extension with Galois group $G$. Then there
is a $G$-equivariant $K$-linear isomorphism
$C(G,K) \to L$.
\end{theorem}

Mostowski~\cite[Satz~(3.2)]{Mostowski}
had earlier obtained this result
when the degrees of all finite subextensions are invertible
in the base field, in particular in characteristic zero.
A cohomological proof of Lenstra's theorem was subsequently
given in \cite{LundstromCohomology}.
Related results include normal bases for infinite Galois
ring extensions \cite{LundstromRing} and self-dual normal
bases for infinite Galois field extensions
\cite{LundstromSelfDual}.

The completely normal basis theorem admits an analogous
reformulation. Suppose that $L/K$ is finite. If $x\in L$ 
and $H$ is a subgroup of $G$, define
$\varphi^H : (H,L^H)\to L$ by
$\varphi^H(f)=\sum_{\sigma\in H}f(\sigma)\sigma(x)$, where
$L^H$ denotes the fixed field of $H$.
Then $x$ is a completely normal basis generator for $L/K$ if and only
if $\varphi^H$ is an isomorphism for every subgroup $H$ of $G$.

This formulation extends naturally to arbitrary Galois extensions. If
$H$ is an open subgroup of $G$, let
$i_H:C(H,K) \to C(G,K)$
be extension by zero. A $G$-equivariant 
$K$-linear map
$\varphi:C(G,K)\to L$ induces, using
$C(H,L^H)\cong L^H\otimes_K C(H,K)$,
where $L^H$ has the discrete topology,
an $H$-equivariant $L^H$-linear map
$\varphi^H:C(H,L^H) \to L$ by
$\varphi^H (a\otimes f)=
a\varphi(i_H(f))$.

Our main result simultaneously generalizes
Theorems~\ref{thm:CNBT} and~\ref{thm:Lenstra}:

\begin{theorem}[Infinite completely normal basis theorem]\label{thm:main-field}
Let $L/K$ be a Galois field extension with Galois group $G$. Then there
exists a $G$-equivariant $K$-linear isomorphism
$\varphi:C(G,K) \to L$
such that, for every open subgroup $H$ 
of $G$, the induced
$H$-equivariant $L^H$-linear map
$\varphi^H:C(H,L^H) \to L$
is an isomorphism.
\end{theorem}

When $K$ is finite, Hachenberger proved the existence of a
trace-coherent family of completely normal basis generators
\cite[Theorem~26.4]{Hachenberger1997}. By
Proposition~\ref{prop:coherent-Lenstra-galois}, this is equivalent
to Theorem~\ref{thm:main-field} in that case.
When $K$ is infinite, we construct such a family by combining
a trace-lifting theorem for completely normal basis generators
with a new inverse-limit theorem.

Section~\ref{sec:preliminaries} collects preliminary results on
Zariski-open subsets, inverse limits, and finite-dimensional algebras.
In Theorem~\ref{thm:cardinal-bounded-inverse-limit}, we prove a
cardinal-bounded variant of the surjective case of Lenstra's
inverse-limit theorem \cite[Proposition~(1.5)(c)]{Lenstra}.
In Section~\ref{sec:finitegaloisring}, we prove that, for finite
Galois ring extensions over an infinite base field, trace maps are
surjective on completely normal basis generators.
Section~\ref{sec:profinitegaloisring} identifies trace-coherent
families of completely normal basis generators with equivariant
isomorphisms of continuous function modules satisfying the open
subgroup condition of Theorem~\ref{thm:main-field}. We then show
that, under cardinality hypotheses, a trace-coherent family
prescribed on a directed subsystem extends to the full system.
We conclude by proving Theorem~\ref{thm:main-field}, treating
finite and infinite base fields separately.
Note that when $K$ is finite, our inverse-limit argument on finite sets provides
an alternative to Hachenberger's explicit construction of trace-coherent
families of completely normal basis generators
\cite[Theorem~26.4]{Hachenberger1997}.

\section{Preliminaries}\label{sec:preliminaries}

Throughout, $K$ denotes a field and $G$ a group with identity $e$.
We write $H\leq G$ and $H\trianglelefteq G$ for subgroups and
normal subgroups, respectively, $|X|$ for the cardinality of
a set $X$, and $\mathbb N=\{1,2,3,\ldots\}$.
If $G$ acts on a set $X$, then we let $X^G$ denote 
the set 
$\{x\in X \mid 
\sigma(x)=x\text{ for all }\sigma\in G\}$.

\subsection{Zariski-open sets}\label{subsec:zariski-limits}

We equip each nonempty finite-dimensional affine $K$-space $A$
with its Zariski topology \cite[Ch.~IX, \S2 and \S5]{Lang},
with basic open sets $D(f)=\{a\in A:f(a)\ne 0\}$ for polynomials
$f$ on $A$. Affine subspaces carry the induced topology,
and affine-linear maps are Zariski continuous.

\begin{lemma}\label{lem:affine-open}
Every surjective affine-linear map between finite-dimensional
affine $K$-spaces is Zariski open.
\end{lemma}

\begin{proof}
Let $q:A\to B$ be such a map and let $O\subseteq A$ be open.
For each $a\in O$, choose an affine-linear section $s$ of $q$
with $s(q(a))=a$. Then $s^{-1}(O)$ is an open neighborhood
of $q(a)$ contained in $q(O)$. Thus $q(O)$ is open.
\end{proof}

\begin{lemma}\label{lem:small-open-intersection}
Let $K$ be infinite and let $A$ be a nonempty finite-dimensional
affine $K$-space. The intersection of any family of fewer than
$|K|$ nonempty Zariski-open subsets of $A$ is nonempty.
\end{lemma}

\begin{proof}
If $\dim A=0$, then $A$ is a singleton. Otherwise, the assertion
follows by taking complements in \cite[Theorem~3.14]{Ghosh2022}.
\end{proof}

In particular, over an infinite field, any two nonempty
Zariski-open subsets of an affine $K$-space meet, 
so every nonempty Zariski-open subset is dense.

\subsection{Inverse limits}

A partially ordered set $I$ is \emph{directed} if it is
nonempty and every pair of elements has a common upper bound.
An \emph{inverse system of sets} indexed by a directed set $I$
consists of sets $X_i$ and maps $p_{ji}:X_j\to X_i$ for
$i\leq j$, satisfying $p_{ii}=\operatorname{id}_{X_i}$ and
$p_{ki}=p_{ji}\circ p_{kj}$ whenever $i\leq j\leq k$.
Its \emph{inverse limit} $\varprojlim_{i\in I}X_i$ 
is the set
$\{ (x_i)_{i\in I}\in\prod_{i\in I}X_i
\mid  p_{ji}(x_j)=x_i \text{ whenever }i\leq j \}$.

The following result will be used to prove
Theorem~\ref{thm:main-field} over finite base fields.

\begin{proposition}\label{prop:inverse-limit-divisibility}
Let $(X_n,T_{m,n})$ be an inverse system of nonempty finite
sets indexed by the positive integers ordered by divisibility.
Then $\varprojlim_n X_n\neq\varnothing$.
\end{proposition}

\begin{proof}
The assertion is a special case of
\cite[Ch.~III, \S7, no.~4, Theorem~1]{BourbakiSets}.
For completeness, we give an elementary proof
adapted to the present setting.
For each $n\geq1$, put
$Y_n=\bigcap_{n\mid m}T_{m,n}(X_m)$.
Choose a multiple $d_n$ of $n$ with
$|T_{d_n,n}(X_{d_n})|$ minimal.
If $n\mid m$ and $r=\operatorname{lcm}(d_n,m)$, then
$
T_{r,n}(X_r)
\subseteq
T_{d_n,n}(X_{d_n})\cap T_{m,n}(X_m)$.
Minimality therefore gives
$T_{r,n}(X_r)=T_{d_n,n}(X_{d_n})$.
Thus,
$Y_n=T_{d_n,n}(X_{d_n})\neq\varnothing$,
and $T_{r,n}(X_r)=Y_n$ whenever $d_n\mid r$.
If $n\mid m$, choose a common multiple $r$ of $d_n$ and $d_m$.
Then $T_{m,n}(Y_m)
=T_{m,n}(T_{r,m}(X_r)) =T_{r,n}(X_r) =Y_n$.
Thus the restricted transition maps are surjective.
We may therefore choose recursively $y_k\in Y_{k!}$ such that
$T_{(k+1)!,k!}(y_{k+1})=y_k$ for every $k\geq1$.
Put $x_n=T_{n!,n}(y_n)$.
If $n\mid m$, the recursive relations give
$T_{m!,n!}(y_m)=y_n$, and hence
$T_{m,n}(x_m) =T_{m!,n}(y_m) =T_{n!,n}(y_n)=x_n$.
Consequently $(x_n)_{n\geq1}\in\varprojlim_n X_n$.
\end{proof}

For an inverse system $(X_i,p_{ji})_{i\in I}$, call a
family $(x_j)_{j\in S}$, where $S\subseteq I$ and
$x_j\in X_j$, \emph{finitely liftable} if, for every
finite subset $F\subseteq S$, there exist an upper
bound $k\in I$ of $F$ and an element $y\in X_k$
such that $p_{kj}(y)=x_j$ for all $j\in F$.

Every finitely liftable family 
$(x_j)_{j\in S}$ is coherent.
Indeed, if $i\leq j$ in $S$ and $y\in X_k$ is a 
common lift of $x_i$ and $x_j$, then
$p_{ji}(x_j)=p_{ji}(p_{kj}(y))=p_{ki}(y)=x_i$.
Conversely, if $S$ is directed, then every coherent 
family on $S$ is finitely liftable. 
For any finite subset of $S$,
the coordinate at an upper bound in $S$ gives a common lift.

\begin{lemma}\label{lem:finite-liftable-extension}
Let $\kappa$ be an infinite cardinal, and let
$(X_i,p_{ji})_{i\in I}$ be an inverse system of
nonempty sets indexed by a directed set $I$
with $|I|\leq\kappa$.
Suppose that, whenever $S\subseteq I$ satisfies
$|S|<\kappa$ and $i\in I\setminus S$, every finitely
liftable family on $S$ extends to a finitely liftable
family on $S\cup\{i\}$.
Then, for every directed subset $J\subseteq I$
with $|J|<\kappa$, the restriction map
$\varprojlim_{i\in I}X_i \to \varprojlim_{j\in J}X_j$
is surjective.
\end{lemma}

\begin{proof}
Fix a directed subset $J\subseteq I$ with $|J|<\kappa$
and a coherent family $(a_j)_{j\in J}$.
Since \(J\) is directed, $(a_j)_{j\in J}$
is finitely liftable. We extend $(a_j)_{j\in J}$
by transfinite recursion,
maintaining finite liftability by transfinite induction;
see \cite[Ch.~III, \S2, no.~2, C59--C60]{BourbakiSets}.

Write $I\setminus J=\{i_\alpha:\alpha<\mu\}$,
where $\mu=|I\setminus J|\leq\kappa$ is regarded
as its initial ordinal, and put
$S_\alpha=J\cup\{i_\beta:\beta<\alpha\}$
for $\alpha\leq\mu$.
To make the choices definite, fix a well-ordering
of $\bigcup_{i\in I}X_i$.
Start with $x_j=a_j$ for $j\in J$.
Suppose that $\alpha<\mu$ and a finitely liftable
family has been constructed on $S_\alpha$.
Since $|J|<\kappa$ and $|\alpha|<\mu\leq\kappa$,
we have $|S_\alpha|<\kappa$.
The hypothesis therefore allows us to choose
$x_{i_\alpha}\in X_{i_\alpha}$ so that the extended
family on $S_{\alpha+1}$ remains finitely liftable.
Choose the least such value in the fixed well-ordering.
At a nonzero limit stage $\lambda\leq\mu$, take
the union of the previously constructed families.
Every finite subset of $S_\lambda$ is contained
in some $S_\alpha$ with $\alpha<\lambda$.
Thus every finite subfamily already has a common
lift, and the union remains finitely liftable.
The recursion produces a finitely liftable, hence 
coherent, family
$(x_i)_{i\in I}$ extending $(a_j)_{j\in J}$.
This proves surjectivity.
\end{proof}

The following result is a cardinal-bounded variant
of the surjective case of Lenstra's inverse-limit
theorem \cite[Proposition~(1.5)(c)]{Lenstra},
which extends a theorem of Bourbaki
\cite[Ch.~III, \S7, no.~4, Theorem~1]{BourbakiSets}.
Recall that a family of subsets has the 
\emph{finite intersection property}
if every finite subfamily has nonempty intersection.

\begin{theorem}
\label{thm:cardinal-bounded-inverse-limit}
Let $\kappa$ be an infinite cardinal, and let
$(X_i,p_{ji})_{i\in I}$ be an inverse system of
nonempty sets with surjective transition maps,
indexed by a directed set $I$ with $|I|\leq\kappa$.
For each $i\in I$, suppose that a collection
$\mathcal C_i$ of subsets of $X_i$ satisfies:
\begin{enumerate}
\item[\rm (i)]
$X_i\in\mathcal C_i$, and $\mathcal C_i$ is closed
under finite intersections.

\item[\rm (ii)]
Every subfamily of $\mathcal C_i$ of cardinality
less than $\kappa$ having the finite intersection
property has nonempty intersection.

\item[\rm (iii)]
For $i\leq j$ and $x\in X_i$, the fibre
$p_{ji}^{-1}(\{x\})$ belongs to $\mathcal C_j$.

\item[\rm (iv)]
For $i\leq j$ and $C\in\mathcal C_j$, the image
$p_{ji}(C)$ belongs to $\mathcal C_i$.
\end{enumerate}
Then, for every directed subset $J\subseteq I$ 
such that $|J|<\kappa$, the restriction map
$\varprojlim_{i\in I}X_i \to
\varprojlim_{j\in J}X_j$
is surjective.
\end{theorem}

\begin{proof}
By Lemma~\ref{lem:finite-liftable-extension}, it
suffices to verify the extension property for
finitely liftable partial families.
Fix such a family $(x_j)_{j\in S}$ with
$|S|<\kappa$, and let $i\in I\setminus S$.

First observe that a finite family having a common
lift in some $X_\ell$ has a common lift in $X_k$
for every upper bound $k$ of its indices.
Indeed, choose $r\geq k,\ell$, lift the given
element from $X_\ell$ to $X_r$ using surjectivity,
and project to $X_k$.

For each finite $F\subseteq S$, choose an upper
bound $k$ of $F\cup\{i\}$. We now define
$C_F=
p_{ki}(
\bigcap_{j\in F}p_{kj}^{-1}(\{x_j\})).$
When $F=\varnothing$, the intersection inside
the parentheses is understood to be $X_k$.
Conditions~(i), (iii), and~(iv) imply that
$C_F\in\mathcal C_i$.

By the preceding observation, $C_F$ is independent
of the chosen upper bound $k$.
It consists precisely of those $u\in X_i$ such
that $u$ and $(x_j)_{j\in F}$ have a common lift.
Finite liftability and the same observation show
that $C_F\neq\varnothing$.
Moreover, for any finite subsets
$F_1,\ldots,F_r\subseteq S$, we have
$C_{F_1\cup\cdots\cup F_r}
\subseteq C_{F_1}\cap\cdots\cap C_{F_r}$.
Thus the family of all $C_F$ has the finite
intersection property.

Since $|S|<\kappa$, there are fewer than $\kappa$
finite subsets of $S$.
Condition~(ii) therefore allows us to choose
$x_i\in
\bigcap_{F\subseteq S, \ F\text{ finite}}C_F$.
In particular, $x_i\in C_{\varnothing}=X_i$.
For every finite $F\subseteq S$, membership in
$C_F$ means that $x_i$ and $(x_j)_{j\in F}$
have a common lift.
Hence adjoining $x_i$ preserves finite liftability.
\end{proof}

The following result will be used to prove
Theorem~\ref{thm:main-field} over infinite base fields.

\begin{proposition}\label{prop:relative-affine-open-limit}
Let $K$ be an infinite field and let $(V_i,p_{ji})$ be an
inverse system of nonempty finite-dimensional affine
$K$-spaces with surjective affine-linear transition maps,
indexed by a directed set $I$ with $|I|\leq|K|$.
Let $U_i\subseteq V_i$ be nonempty Zariski-open subsets
satisfying $p_{ji}(U_j)=U_i$ whenever $i\leq j$.
Then, for every directed subset $J\subseteq I$ with
$|J|<|K|$, the restriction map
$\varprojlim_{i\in I}U_i \to \varprojlim_{j\in J}U_j$
is surjective.
\end{proposition}

\begin{proof}
Put $\kappa=|K|$. We apply
Theorem~\ref{thm:cardinal-bounded-inverse-limit}
to the inverse system $(U_i,p_{ji}|_{U_j})$.
For each $i\in I$, let $\mathcal C_i$ consist of
$\varnothing$ and all nonempty subsets of $U_i$
that are relatively Zariski open in some affine
subspace of $V_i$.

Clearly, $U_i\in\mathcal C_i$.
If finitely many members of $\mathcal C_i$ have
nonempty intersection, their intersection is relatively
Zariski open in the intersection of the corresponding
affine subspaces.
Since $\varnothing\in\mathcal C_i$, it follows that
$\mathcal C_i$ is closed under finite intersections.

For $i\leq j$ and $x\in U_i$, the fibre of
$p_{ji}|_{U_j}$ over $x$ is
$U_j\cap p_{ji}^{-1}(\{x\})$,
which is relatively Zariski open in the affine
subspace $p_{ji}^{-1}(\{x\})$ of $V_j$.
Hence this fibre belongs to $\mathcal C_j$.
Moreover, if $C\in\mathcal C_j$ is nonempty and
relatively Zariski open in an affine subspace
$B\subseteq V_j$, then
Lemma~\ref{lem:affine-open}, applied to
$p_{ji}|_B:B \to p_{ji}(B)$,
shows that $p_{ji}(C)$ is relatively Zariski open
in the affine subspace $p_{ji}(B)$.
Since $p_{ji}(C)\subseteq U_i$, we have
$p_{ji}(C)\in\mathcal C_i$.
The image of the empty set also belongs to
$\mathcal C_i$.

It remains to verify the intersection property.
Fix $i\in I$, and let
$\{C_\lambda:\lambda\in\Lambda\}\subseteq\mathcal C_i$
have the finite intersection property, with
$|\Lambda|<\kappa$.
The case $\Lambda=\varnothing$ follows from
$U_i\neq\varnothing$, so assume $\Lambda\neq\varnothing$.
For each $\lambda$, choose an affine subspace
$A_\lambda\subseteq V_i$ in which $C_\lambda$
is relatively Zariski open.

For every nonempty finite $F\subseteq\Lambda$, put
$A_F=\bigcap_{\lambda\in F}A_\lambda$.
These affine subspaces are nonempty by the finite
intersection property.
Choose a nonempty finite $F_0\subseteq\Lambda$
for which $\dim A_{F_0}$ is minimal, and put
$A=A_{F_0}$.
For every $\lambda\in\Lambda$,  
$A_{F_0\cup\{\lambda\}} = A\cap A_\lambda\subseteq A$.
Minimality of dimension gives
$A\cap A_\lambda=A$, so $A\subseteq A_\lambda$.
Consequently, $C_\lambda\cap A$ is relatively
Zariski open in $A$.
It is also nonempty, since it contains
$\bigcap_{\nu\in F_0\cup\{\lambda\}}C_\nu$,
which is nonempty by the finite intersection property.

Lemma~\ref{lem:small-open-intersection} now gives
$\varnothing\neq
\bigcap_{\lambda\in\Lambda}(C_\lambda\cap A)
\subseteq
\bigcap_{\lambda\in\Lambda}C_\lambda$.
All hypotheses of
Theorem~\ref{thm:cardinal-bounded-inverse-limit}
are therefore satisfied, and the desired
surjectivity follows.
\end{proof}

\begin{remark}\label{rem:relative-limit-counterexample}
The assumption $|J|<|K|$ in
Proposition~\ref{prop:relative-affine-open-limit}
cannot be replaced by $|J|\leq|K|$.
We adapt Waterhouse's example \cite{Waterhouse1972}.
Let $K$ be an infinite field, put $E=K\sqcup\{*\}$,
and let $I$ be the set of finite subsets of $E$,
ordered by inclusion.
Put $J=\{F\in I \mid F\subseteq K\}$.
Then $I$ and $J$ are directed and
$|I|=|J|=|K|$.
For $F\in I$, put $V_F=K^F$ and
$U_F=\{x\in K^F:x_a\neq x_b
       \text{ whenever }a,b\in F,\ a\neq b\}$.
These sets are nonempty and Zariski open.
The coordinate restriction maps are surjective linear
maps between the $V_F$ and are also surjective on
the $U_F$: every injection from a finite subset of
$E$ into $K$ extends to any larger finite subset,
since $K$ is infinite.
However, the coherent family
$(\operatorname{id}_K|_F)_{F\in J}$
cannot extend to an element of
$\varprojlim_{F\in I}U_F$.
Such an extension would define an injection $E\to K$
whose restriction to $K$ is the identity, leaving
no possible value for the additional element $*$.
\end{remark}

\subsection{Algebras}

All $K$-algebras are assumed to be associative and unital,
but need not be commutative.
Their multiplicative identities are denoted by $1$.
A $K$-algebra homomorphism is a $K$-linear,
multiplicative map preserving $1$.
If $R$ is a unital $K$-subalgebra of $T$, we call
$T$ a $K$-algebra extension of $R$ and write $T/R$.
The group of units of $R$ is denoted by $R^\times$.

\begin{lemma}\label{lem:units-open}
Let $A$ be a finite-dimensional $K$-algebra.
Then $A^\times$ is a nonempty basic Zariski-open
subset of the affine $K$-space $A$.
\end{lemma}

\begin{proof}
For $a\in A$, let $L_a:A\to A$ denote left
multiplication by $a$.
An element $a$ is invertible if and only if $L_a$
is invertible. Indeed, if $L_a$ is invertible, choose
$b\in A$ with $ab=1$; injectivity of $L_a$ then gives
$ba=1$.
The function $d:A\to K$, defined by
$d(a)=\det(L_a)$, is polynomial in linear coordinates.
Thus $A^\times=D(d)$, which is nonempty since $d(1)=1$.
\end{proof}

\begin{proposition}\label{prop:unit-lifting-zariski}
Let $R$ and $S$ be finite-dimensional $K$-algebras,
and let $f:R\to S$ be a surjective $K$-algebra
homomorphism. Then $f(R^\times)=S^\times$.
\end{proposition}

\begin{proof}
This holds more generally for surjective homomorphisms
from Artinian rings; see \cite[Lemma~(2.6)]{Lenstra}.
For completeness, we give a direct proof in two cases.
The inclusion $f(R^\times)\subseteq S^\times$ is immediate. Now we show the reverse inclusion.

Case 1: $K$ is infinite.
By Lemma~\ref{lem:units-open}, $R^\times$ is a
nonempty Zariski-open subset of $R$.
Since $f$ is surjective and $K$-linear,
Lemma~\ref{lem:affine-open} shows that
$H:=f(R^\times)$ is a nonempty Zariski-open subset
of $S$. Moreover, $H$ is a subgroup of $S^\times$.
Take $s\in S^\times$. Left multiplication by $s$
is a $K$-linear automorphism of $S$, so $sH$
is also nonempty and Zariski open in $S$.
Since $K$ is infinite, $H\cap sH\neq\varnothing$.
Thus $s=h_1h_2^{-1}\in H$ for some $h_1,h_2\in H$.
This proves $f(R^\times) \supseteq S^\times$.

Case 2: $K$ is finite.
Let $s\in S^\times$ and choose $a\in R$ with $f(a)=s$.
Since $R$ is finite-dimensional over the finite field
$K$, it is finite. The powers of $a$ are therefore
eventually periodic.
Taking $n$ sufficiently large and divisible by the
period gives $a^{2n}=a^n$.
Thus $e:=a^n$ is idempotent.
Since $f(e)=s^n$ is both a unit and an idempotent,
we have $f(e)=1$.
Put $u=ae+(1-e)$.
Since $a$ commutes with $e$, the two summands
annihilate one another, and
$u^n=(ae)^n+(1-e)=e+(1-e)=1$.
Hence $u\in R^\times$, and
$f(u)=f(a)f(e)+1-f(e)=s$.
This proves $f(R^\times) \supseteq S^\times$.
\end{proof}

\section{Finite Galois ring extensions}\label{sec:finitegaloisring}

Let $G$ be a finite group and let $S/R$ be an extension
of commutative $K$-algebras.
Following \cite{ChaseHarrisonRosenberg}, we call $S/R$
\emph{$G$-Galois} if $S$ is equipped with an action
of $G$ by $R$-algebra automorphisms, $S^G=R$, and
the canonical map
\[
S\otimes_R S \ni (a \otimes b) \mapsto 
(a\,\sigma(b))_{\sigma\in G} \in
\prod_{\sigma\in G}S
\]
is an isomorphism of $S$-algebras.
Throughout this section, we assume that $S/R$ is
$G$-Galois and that $R$ is finite-dimensional over $K$.

\begin{proposition}\label{prop:galois-finite-dimensional}
The $K$-algebra $S$ is finite-dimensional.
\end{proposition}

\begin{proof}
By \cite[Theorem~1.3(c)]{ChaseHarrisonRosenberg},
$S$ is finitely generated projective as an $R$-module.
Since $R$ is finite-dimensional over $K$, so is $S$.
\end{proof}

An element $x\in S$ is called a
\emph{$G$-normal basis generator} for $S/R$ if
the family $(\sigma(x))_{\sigma\in G}$ is an
$R$-basis of $S$.
Such a basis is called a \emph{$G$-normal basis}
for $S/R$.

\begin{theorem}
\label{thm:normalbasissemilocal}
The extension $S/R$ has a $G$-normal basis.
\end{theorem}

\begin{proof}
Since $R$ is finite-dimensional over $K$, it is
Artinian and hence semilocal.
The assertion therefore follows from
\cite[Theorem~4.2(c)]{ChaseHarrisonRosenberg}.
\end{proof}

For $x\in S$, define its \emph{$G$-resolvent} by
$\rho_G(x)=
\sum_{\sigma\in G}\sigma(x)\sigma^{-1}\in S[G]$,
where $S[G]$ denotes the ordinary group algebra.

\begin{lemma}\label{lem:Mazur-resolvent}
An element $x\in S$ is a $G$-normal basis
generator for $S/R$ if and only if
$\rho_G(x)\in S[G]^\times$.
\end{lemma}

\begin{proof}
This is a special case of
\cite[Proposition~1]{Mazur}.
\end{proof}

Let $H\leq G$. By
\cite[Theorem~2.2]{ChaseHarrisonRosenberg},
$S/S^H$ is an $H$-Galois extension.
If $H\trianglelefteq G$, then $S^H/R$ is a
$(G/H)$-Galois extension.
Let $\mathcal N_G(S/R)$ denote the set of
$G$-normal basis generators for $S/R$, and define
$\Tr_H:S \to S^H$ by
$\Tr_H(x)=\sum_{\sigma\in H}\sigma(x)$.
We use the action of $R[G]$ on $S$, given by
$\left(\sum_{\sigma\in G}r_\sigma\sigma\right)(x)
=\sum_{\sigma\in G}r_\sigma\sigma(x)$.

\begin{proposition}\label{prop:normal-trace-lift}
Let $H\leq G$. Then the trace map
$\Tr_H:S\to S^H$ is surjective.
If $H\trianglelefteq G$, then
$\Tr_H(\mathcal N_G(S/R)) =
\mathcal N_{G/H}(S^H/R)$.
In particular, $\Tr_G(\mathcal N_G(S/R))=R^\times$.
\end{proposition}

\begin{proof}
Fix $x\in\mathcal N_G(S/R)$, which exists by
Theorem~\ref{thm:normalbasissemilocal}.
Let $T$ be a set of representatives for the right cosets of $H$ in $G$.
Given $s\in S^H$, write
$s=\sum_{\sigma\in G} a_\sigma\,\sigma(x)$,
$a_\sigma\in R$.
Since $h(s)=s$ for every $h\in H$, uniqueness of the normal basis
expansion gives $a_{h\sigma}=a_\sigma$. Hence
$s=\sum_{\tau\in T}a_\tau\sum_{h\in H}(h\tau)(x)
 =\Tr_H\left(\sum_{\tau\in T}a_\tau\,\tau(x)\right)$.
This proves surjectivity of $\Tr_H$ 
and shows that the elements
$\Tr_H(\tau(x))$, $\tau\in T$, span $S^H$ over $R$.
They are linearly independent because in any relation among them,
the coefficient of $\tau(x)$ is the coefficient of
$\Tr_H(\tau(x))$ in that relation.

Now suppose that $H\trianglelefteq G$, and put $t=\Tr_H(x)$.
Since $\Tr_H$ commutes with the action of $G$, the basis just
constructed consists of the $G/H$-conjugates of $t$.
Thus $t\in\mathcal N_{G/H}(S^H/R)$. As $x$ was arbitrary, this gives
$\Tr_H(\mathcal N_G(S/R))\subseteq\mathcal 
N_{G/H}(S^H/R)$.
Conversely, let $z\in\mathcal N_{G/H}(S^H/R)$ and write $z=v(t)$
for the unique $v\in R[G/H]$ with this property.
Since $a\mapsto a(t)$ and $a\mapsto a(z)=(av)(t)$ are both
isomorphisms from $R[G/H]$ to $S^H$, right multiplication by $v$
is bijective. Hence $v\in R[G/H]^\times$.
By Proposition~\ref{prop:unit-lifting-zariski}, the canonical
surjection $R[G]\to R[G/H]$ lifts $v$ to a unit
$u\in R[G]^\times$. Put $y=u(x)$. The map
$R[G]\to S$, $a\mapsto a(y)=(au)(x)$,
is an isomorphism, so $y\in\mathcal N_G(S/R)$.
The $G$-equivariance of the trace and the action of $G$ on $S^H$
through $G/H$ give $\Tr_H(y)=v(t)=z$. This proves the reverse inclusion.

Taking $H=G$ gives the final assertion, since
$\mathcal N_{\{e\}}(R/R)=R^\times$.
\end{proof}

\begin{lemma}\label{lem:nonemptyzariskinb}
$\mathcal N_G(S/R)$ is a nonempty
Zariski-open subset of the affine $K$-space $S$.
\end{lemma}

\begin{proof}
Nonemptiness follows from
Theorem~\ref{thm:normalbasissemilocal}.
Since $S[G]$ is finite-dimensional over $K$,
Lemma~\ref{lem:units-open} shows that $S[G]^\times$
is Zariski open.
The map $\rho_G:S\to S[G]$ is $K$-linear and
hence Zariski continuous. Thus
Lemma~\ref{lem:Mazur-resolvent} gives
$\mathcal N_G(S/R)=\rho_G^{-1}(S[G]^\times)$,
which is Zariski open in $S$.
\end{proof}

An element $x\in S$ is called a 
$G$-completely normal basis generator
for $S/R$ if $x \in \mathcal{N}_H(S/S^H)$
for every $H \leq G$. In that case we say
that $(\sigma(x))_{\sigma \in G}$ is
a $G$-completely normal basis for 
$S/R$. We put
$\mathcal{CN}_G(S/R) = \bigcap_{H \leq G}
\mathcal{N}_H(S/S^H)$.

\begin{theorem}\label{thm:completenormalbasis}
Let $K$ be infinite.
Then $\mathcal{CN}_G(S/R)$ is a nonempty
Zariski-open subset of the affine $K$-space $S$.
In particular, $S/R$ has a $G$-completely normal basis.
\end{theorem}

\begin{proof}
For every subgroup $H\leq G$, the algebra $S^H$
is finite-dimensional over $K$, since it is a
$K$-subspace of $S$.
Lemma~\ref{lem:nonemptyzariskinb}, applied to the
$H$-Galois extension $S/S^H$, therefore shows that
$\mathcal{N}_H(S/S^H)$ is nonempty and Zariski open
in $S$.
Since $K$ is infinite, these subsets are dense.
As $G$ has only finitely many subgroups, their
intersection $\mathcal{CN}_G(S/R)$ is nonempty
and Zariski open.
\end{proof}

\begin{example}\label{ex:finite-field-no-complete-normal-basis}
The hypothesis that $K$ is infinite in
Theorem~\ref{thm:completenormalbasis} cannot be omitted.
Let $K=\mathbb F_3$, put $S=K\times K$, and let
$G=C_2=\{e,\tau\}$ act by $\tau(a,b)=(b,a)$.
Identifying $K$ with the diagonal subalgebra of $S$,
the extension $S/K$ is $G$-Galois.
If $x=(a,b)$ were a $G$-completely normal basis
generator, the condition for the trivial subgroup
would imply $x\in S^\times$, so $a,b\neq0$.
But then $\left|\begin{smallmatrix}a&b\\ b&a\end{smallmatrix}\right|
=a^2-b^2 = 1-1 = 0$,
so $(x,\tau(x))$ is not a $K$-basis of $S$.
Thus $\mathcal{CN}_G(S/K)=\varnothing$.
\end{example}

\begin{lemma}\label{lem:trace-preserves-complete-normality}
Let $H\trianglelefteq G$. Then
$\Tr_H(\mathcal{CN}_G(S/R))
\subseteq \mathcal{CN}_{G/H}(S^H/R)$.
\end{lemma}

\begin{proof}
Let $x\in\mathcal{CN}_G(S/R)$ and put $y=\Tr_H(x)$.
For every subgroup $V$ with $H\leq V\leq G$, we have
$x\in\mathcal{N}_V(S/S^V)$.
Since $H\trianglelefteq V$,
Proposition~\ref{prop:normal-trace-lift}, applied to
$S/S^V$, gives $y\in\mathcal{N}_{V/H}(S^H/S^V)$.
Every subgroup of $G/H$ has the form $V/H$, and
$(S^H)^{V/H}=S^V$.
Thus $y\in\mathcal{CN}_{G/H}(S^H/R)$.
\end{proof}

\begin{lemma}\label{lem:normality-in-trace-fiber}
Suppose that $\mathcal{CN}_G(S/R)\ne\varnothing$.
Let $H\trianglelefteq G$ and $U\leq G$, and put $J=UH$.
Then every $x\in\mathcal N_{J/H}(S^H/S^J)$ has a lift
$y\in\mathcal N_U(S/S^U)$ with $\Tr_H(y)=x$.
\end{lemma}

\begin{proof}
Put $I=U\cap H$. The isomorphism $U/I\cong J/H$ and the
inclusion $S^H\subseteq S^I$ induce a unital homomorphism
$S^H[J/H]\to S^I[U/I]$
sending $\rho_{J/H}(x)$ to $\rho_{U/I}(x)$.
By Lemma~\ref{lem:Mazur-resolvent}, the former resolvent is a
unit, so the latter is a unit as well. Thus
$x\in\mathcal N_{U/I}(S^I/S^U)$.
Choose $s\in\mathcal{CN}_G(S/R)$. By
Proposition~\ref{prop:normal-trace-lift},
$\Tr_U(s)\in(S^U)^\times$ and $\Tr_J(s)\in(S^J)^\times$.
Since $S^J\subseteq S^U$, we have
$c=\Tr_J(s)^{-1}\Tr_U(s)\in(S^U)^\times$.
Hence $cx\in\mathcal N_{U/I}(S^I/S^U)$.
Applying Proposition~\ref{prop:normal-trace-lift} to
$S/S^U$ and $I\trianglelefteq U$, choose
$y\in\mathcal N_U(S/S^U)$ with $\Tr_I(y)=cx$.
Let $T\subseteq H$ represent the left cosets of $I$ in $H$.
Since $J=HU$ and $I=H\cap U$, the same set represents the
left cosets of $U$ in $J$. Both $x$ and $\Tr_J(s)$ are
fixed by $H$, so
$\Tr_H(y) =\sum_{\tau\in T}\tau(cx) =x\,\Tr_J(s)^{-1}
  \sum_{\tau\in T}\tau(\Tr_U(s)) =x$.
\end{proof}

\begin{theorem}\label{thm:complete-normal-trace-lift}
If $K$ is infinite and $H\trianglelefteq G$,
then $\Tr_H(\mathcal{CN}_G(S/R)) 
=\mathcal{CN}_{G/H}(S^H/R)$.
\end{theorem}

\begin{proof}
By Lemma~\ref{lem:trace-preserves-complete-normality}, it
suffices to show that every 
$x\in\mathcal{CN}_{G/H}(S^H/R)$
has a completely normal lift. Fix such an $x$ and put
$\mathcal T_x=\Tr_H^{-1}(\{x\})$.
The trace is $K$-linear and surjective by
Proposition~\ref{prop:normal-trace-lift}, and $S$ is
finite-dimensional over $K$ by
Proposition~\ref{prop:galois-finite-dimensional}.
Thus $\mathcal T_x$ is a nonempty 
finite-dimensional affine $K$-space.
By Theorem~\ref{thm:completenormalbasis},
$\mathcal{CN}_G(S/R)\ne\varnothing$.
For each subgroup $U\leq G$, put $J=UH$ and
$\mathcal O_U=\mathcal T_x\cap\mathcal N_U(S/S^U)$.
Lemma~\ref{lem:nonemptyzariskinb}, applied to the $U$-Galois
extension $S/S^U$, shows that $\mathcal O_U$ is Zariski open
in $\mathcal T_x$. Since $(S^H)^{J/H}=S^J$, complete
normality of $x$ gives $x\in\mathcal N_{J/H}(S^H/S^J)$.
Hence Lemma~\ref{lem:normality-in-trace-fiber} shows that
$\mathcal O_U$ is nonempty.
Since $G$ has finitely many subgroups and $K$ is infinite,
Lemma~\ref{lem:small-open-intersection} gives
$\bigcap_{U\leq G}\mathcal O_U\ne\varnothing$.
Any element of this intersection belongs to
$\mathcal{CN}_G(S/R)$ and has $H$-trace $x$.
\end{proof}

\section{Profinite Galois ring extensions}
\label{sec:profinitegaloisring}

Throughout this section, $R$ and $S$ denote commutative
$K$-algebras, with $R$ finite-dimensional over $K$,
and $G$ denotes a profinite group.
Let $\mathcal U$ be the directed set of open normal
subgroups of $G$, ordered by reverse inclusion.

We call $S/R$ a \emph{profinite $G$-Galois extension}
if $S$, equipped with the discrete topology, has a
continuous action of $G$ by $R$-algebra automorphisms
such that $S^G=R$ and, for every $N\in\mathcal U$,
the extension $S^N/R$ is $(G/N)$-Galois under the
induced action.
Throughout the rest of this section, assume that
$S/R$ is profinite $G$-Galois.
Continuity implies
$S=\bigcup_{N\in\mathcal U}S^N$,
so these finite Galois levels exhaust $S$.

The reformulation of complete normality in
Section~\ref{sec:introduction} extends to this setting.
Let $C(G,R)$ denote the $R$-module of continuous
functions $G\to R$, where $R$ is discrete, with
the translation action
$(\sigma f)(\tau)=f(\sigma^{-1}\tau)$, for
$\sigma,\tau\in G$.
Let $\varphi:C(G,R)\to S$ be a $G$-equivariant
$R$-linear map, and let $H\leq G$ be open.
Since $H$ is also closed, extension by zero defines
an $H$-equivariant map
$i_H:C(H,R)\to C(G,R)$.
With $S^H$ given the discrete topology, there is
a canonical identification
\[
S^H\otimes_R C(H,R)\cong C(H,S^H),
\qquad
a\otimes f \mapsto \bigl(h\mapsto af(h)\bigr).
\]
This follows by expressing continuous functions
using finite clopen partitions of $H$.
Thus $\varphi$ induces an $H$-equivariant
$S^H$-linear map
\[
\varphi^H:C(H,S^H)\to S,
\qquad
\varphi^H(a\otimes f)=a\varphi(i_H(f)).
\]
The natural profinite analogue of complete normality
is therefore the existence of a $G$-equivariant
$R$-linear isomorphism $\varphi:C(G,R)\to S$
such that $\varphi^H$ is an isomorphism for every
open subgroup $H\leq G$.

We next express this condition in terms of completely
normal generators at the finite Galois levels.
For $M,N\in\mathcal U$ with $M\subseteq N$,
the extension $S^M/S^N$ is $(N/M)$-Galois.
We write $\Tr_{S^M/S^N}(x)
=\sum_{\tau M\in N/M}\tau(x)$, for $x\in S^M$.
A family $(x_N)_{N\in\mathcal U}$, with $x_N\in S^N$,
is called \emph{trace coherent} if
$\Tr_{S^M/S^N}(x_M)=x_N$
for $M,N\in\mathcal U$ and $M\subseteq N$.
For a finite set $E$ and a commutative ring $A$,
we write $(E,A)$ for the $A$-module of all functions
$E\to A$.

\begin{proposition}\label{prop:coherent-Lenstra-galois}
There is a natural bijection between trace-coherent
families $(x_N)_{N\in\mathcal U}$ satisfying
$x_N\in\mathcal{CN}_{G/N}(S^N/R)$, for 
$N\in\mathcal U$,
and $G$-equivariant $R$-linear isomorphisms
$\varphi:C(G,R)\to S$ such that $\varphi^H$ is an
isomorphism for every open subgroup $H\leq G$.
The correspondence is characterized by
$x_N=\varphi(1_N)$, where $1_N$ is the
characteristic function of $N$.
\end{proposition}

\begin{proof}
We first establish the correspondence for ordinary
normal basis generators.
Identify $(G/N,R)$ with $C(G,R)^N$ by inflation.
Given a trace-coherent family $(x_N)_{N\in\mathcal U}$
of normal basis generators, define
$\varphi_N : (G/N,R) \to  S^N$ by
$\varphi_N(f)
=\sum_{\sigma N\in G/N}f(\sigma N)\sigma(x_N)$.
Each $\varphi_N$ is a $(G/N)$-equivariant
$R$-linear isomorphism.
For $M\subseteq N$, grouping the terms in
$\varphi_M(f)$ according to the cosets of $N$
and using trace coherence gives
$\varphi_M(f)=\varphi_N(f)$ for
$f \in (G/N,R)$.
Since $C(G,R)$ and $S$ are the unions of these
finite-level function modules and fixed rings,
respectively, the maps $\varphi_N$ induce a
$G$-equivariant $R$-linear isomorphism
$\varphi:C(G,R)\to S$.

Conversely, let $\varphi:C(G,R)\to S$ be such an
isomorphism and put $x_N=\varphi(1_N)$.
Taking $N$-fixed points gives a $(G/N)$-equivariant
$R$-linear isomorphism
$(G/N,R) \to S^N$,
which sends the characteristic-function basis to
$(\sigma(x_N))_{\sigma N\in G/N}$.
Thus $x_N$ is a normal basis generator.
For $M\subseteq N$, the identity
$1_N=\sum_{\tau M\in N/M}\tau 1_M$
and equivariance give
$x_N=\sum_{\tau M\in N/M}\tau(x_M)
=\Tr_{S^M/S^N}(x_M)$.
Hence the family is trace coherent.
These constructions are inverse to each other.

It remains to characterize complete normality.
Let $H\leq G$ be open and let $N\in\mathcal U$
satisfy $N\subseteq H$.
The restriction of $\varphi^H$ to $N$-fixed points is
$(H/N,S^H) \to S^N$, where
$f \mapsto \sum_{hN\in H/N}f(hN)h(x_N)$.
Indeed, $\varphi^H$ sends the characteristic function
of $hN$ to $h(x_N)$.
This restriction is an isomorphism precisely when
$x_N$ is an $(H/N)$-normal basis generator for
$S^N/S^H$.

If every $x_N$ is completely normal, all these
finite-level restrictions are isomorphisms.
As $N$ ranges over the open normal subgroups of $G$
contained in $H$, their domains and codomains
exhaust $C(H,S^H)$ and $S$, respectively.
Thus $\varphi^H$ is an isomorphism.

Conversely, if every $\varphi^H$ is an isomorphism,
so is its restriction to $N$-fixed points.
Every subgroup of $G/N$ has the form $H/N$ for an
open subgroup $H$ containing $N$, and
$(S^N)^{H/N}=S^H$.
Consequently, each $x_N$ is completely normal.
\end{proof}

\begin{remark}\label{rem:profiniteLenstra}
The same construction gives a natural bijection
between trace-coherent families of ordinary normal
basis generators at the finite Galois levels and
$G$-equivariant $R$-linear isomorphisms
$C(G,R)\to S$.
For Galois field extensions, these are the
isomorphisms appearing in Lenstra's infinite
normal basis theorem.
\end{remark}

\begin{theorem}\label{thm:profinite-galois}
Suppose that $K$ is infinite and
$|\mathcal U|\leq |K|$.
Let $\mathcal J\subseteq\mathcal U$ be directed
with $|\mathcal J|<|K|$.
Then every trace-coherent family
$(x_N)_{N\in\mathcal J}$ with
$x_N\in\mathcal{CN}_{G/N}(S^N/R)$, for $N\in\mathcal J$,
extends to a trace-coherent family of completely
normal basis generators indexed by all of
$\mathcal U$.
\end{theorem}

\begin{proof}
For $N\in\mathcal U$, put $V_N=S^N$ and
$U_N=\mathcal{CN}_{G/N}(S^N/R)$.
Each $V_N$ is finite-dimensional over $K$, and
Theorem~\ref{thm:completenormalbasis} shows that
$U_N$ is a nonempty Zariski-open subset of $V_N$.
For $M\subseteq N$, put $p_{MN}=\Tr_{S^M/S^N}$.
These maps are $K$-linear and form an inverse system
by transitivity of trace. Applied to the $(G/M)$-Galois
extension $S^M/R$ and the normal subgroup $N/M$,
Proposition~\ref{prop:normal-trace-lift} and
Theorem~\ref{thm:complete-normal-trace-lift} give
$p_{MN}(V_M)=V_N$, $p_{MN}(U_M)=U_N$,
respectively.
All hypotheses of
Proposition~\ref{prop:relative-affine-open-limit}
are therefore satisfied, and its conclusion
gives the required extension.
\end{proof}

\subsubsection*{Proof of Theorem~\ref{thm:main-field}}

Let $G=\operatorname{Gal}(L/K)$, and let $\mathcal U$
be the set of open normal subgroups of $G$.
The extension $L/K$ is profinite $G$-Galois.
By Proposition~\ref{prop:coherent-Lenstra-galois},
it suffices to construct a trace-coherent family
$(x_N)_{N\in\mathcal U}$ with
$x_N\in\mathcal{CN}_{G/N}(L^N/K)$ for $N\in\mathcal U$.

Case 1: $K$ is finite.
Write $|K|=q$ and, for each $n \in \mathbb{N}$, set
$L_n=\{a\in L \mid a^{q^n}=a\}$ and 
$X_n=\mathcal{CN}_{\operatorname{Gal}(L_n/K)}(L_n/K)$.
Each $L_n/K$ is a finite Galois extension.
Thus Theorem~\ref{thm:CNBT} shows that $X_n$
is nonempty, and it is finite because $L_n$ is finite.
If $n\mid m$, then $L_n\subseteq L_m$, and
Lemma~\ref{lem:trace-preserves-complete-normality}
shows that the trace defines a map
$T_{m,n}:=\Tr_{L_m/L_n}:X_m \to X_n$.
These maps form an inverse system by transitivity
of trace.
Proposition~\ref{prop:inverse-limit-divisibility}
thus gives a family $(a_n)_{n\geq1}$ with
$a_n\in X_n$ and $T_{m,n}(a_m)=a_n$ when $n\mid m$.
Every finite intermediate field $E/K$ of degree $d$
equals $L_d$, since its $q^d$ elements exhaust
the roots of $T^{q^d}-T$.
For $N\in\mathcal U$, put
$d_N=[L^N:K]$ and $x_N=a_{d_N}$.
Then $L^N=L_{d_N}$, so $x_N$ is completely normal
for $L^N/K$.
If $M\subseteq N$, then $d_N\mid d_M$, and hence
$\Tr_{L^M/L^N}(x_M) =T_{d_M,d_N}(a_{d_M})
=a_{d_N}=x_N$.
Thus $(x_N)_{N\in\mathcal U}$ is the required family.

Case 2: $K$ is infinite.
Since $L/K$ is algebraic, counting roots of
polynomials in $K[T]$ gives $|L|=|K|$.
Every finite subextension of $L/K$ is generated
over $K$ by finitely many elements of $L$, so
there are at most $|K|$ such subextensions.
By Galois correspondence, it follows that
$|\mathcal U|\leq|K|$.
Apply Theorem~\ref{thm:profinite-galois} with
$R=K$, $S=L$, $\mathcal J=\{G\}$, and $x_G=1$.
Since
$\mathcal{CN}_{\{e\}}(K/K)=K^\times$,
this gives the required global trace-coherent
family of completely normal basis generators.
\qed

\subsection*{Use of AI tools}
The author used AI tools to assist with language editing.

\end{document}